\documentclass[12pt, reqno]{amsart}
\usepackage{amsmath, amsthm, amscd, amsfonts, amssymb, graphicx, color}
\usepackage[bookmarksnumbered, colorlinks, plainpages]{hyperref}
\hypersetup{colorlinks=true,linkcolor=red, anchorcolor=green, citecolor=cyan, urlcolor=red, filecolor=magenta, pdftoolbar=true}

\newtheorem{theorem}{Theorem}[section]
\newtheorem{lemma}[theorem]{Lemma}
\newtheorem{prop}[theorem]{Proposition}
\newtheorem{cor}[theorem]{Corollary}
\theoremstyle{definition}

\theoremstyle{remark}

\numberwithin{equation}{section}
\begin{document}
		\title[Advancement of numerical radius inequalities of operator and product of operators ]  {Advancement of numerical radius inequalities of operator and product of operators }
	\author[R.K. Nayak]{  Raj Kumar Nayak}

	\address{(Nayak) Department of Mathematics, SRM University-AP, Amaravati, India}
	\email{rajkumarju51@gmail.com}
	\subjclass[2010]{Primary 47A12, 47A30, Secondary  47A63, 15A60}
	\keywords{ Numerical radius; Norm inequality}
	\date{}
	\maketitle 
	\begin{abstract}
		In this article, we proved upper bounds for numerical radius of bounded linear operator and product of operators which generalize and improve existing inequalities. We also obtain a numerical radius inequality of invertible operator using Kantorovich's ratio. 
		\end{abstract}
		\section{Introduction}
		Let $\mathcal{B}({\mathcal{H}})$ be the $C^*$-algebra of all bounded linear operator on a complex Hilbert space $\mathcal{H}.$ Let $\mathcal{B}(\mathcal{H})^{-1}$ be the set of all invertible operators on the complex Hilbert space $\mathcal{H}$. The absolute value of $T$ is denoted by $|T|$, defined by $|T| = (T^*T)^{\frac{1}{2}},$ where $T^*$ is the Hilbert adjoint of the operator $T.$ The numerical range of $T \in \mathcal{B}(\mathcal{H})$ is denoted by $W(T)$, is the image of the unit sphere of $\mathcal{H}$ under the mapping $x \rightarrowtail \langle Tx, x \rangle.$ For $T \in \mathcal{B}(\mathcal{H})$, let 
		\begin{eqnarray*}
		 w(T) &=& \sup_{\|x\|=1} \left|\langle Tx, x \rangle \right|\\
			 \|T\| &=& \sup_{\|x\|=1} \|Tx\|
		\end{eqnarray*}
		 denoted by  numerical radius and  operator norm respectively. It is well known that $w(\cdot)$ forms a norm on $\mathcal{B}(\mathcal{H})$ which is equivalent to the usual operator norm by the following inequality: 
		\begin{equation}\label{eq 1}
		 \frac{\|T\|}{2} \leq w(T) \leq \|T\|.
		 \end{equation}
		  Many improvements of this inequality are developed in recent years. We state few of those improvements: \\ In \cite{kstudiamath1}, Kittaneh proved if $T \in \mathcal{B}(\mathcal{H})$, then 
		  \begin{equation} \label{eq 2}
		  w(T) \leq \frac{1}{2} \left\||T| + |T^*|\right\|. 
		  \end{equation}
		   This inequality was generalized later by El-Haddad and Kittaneh in \cite{kstudiamath3} which says that, if $T \in \mathcal{B}(\mathcal{H})$ then 
		   \begin{equation}\label{eq el kit}
		   w^{2r}(T) \leq \frac{1}{2} \left\||T|^{2r} + |T^*|^{2r}\right\|.
		   \end{equation}
		   		   Abu-Omar et. al. in \cite{abu omar} proved that for $T \in \mathcal{B}(\mathcal{H}),$
		   		   \begin{equation}\label{eq 3}
		   		   w^2(T) \leq \frac{1}{4} \left\||T|^2 + |T^*| \right\| + \frac{1}{2} w(T^2).
		   		   \end{equation}
		   		   In 2021, Bhunia et. al. in \cite{bulletin sci} proved that, for $T \in \mathcal{B}(\mathcal{H})$
		   		   \begin{equation}\label{eq 11}
		   		   w^2(T) \leq \frac{1}{4} \left\||T|^2 + |T^*|^2 \right\| + \frac{1}{2} w\left(|T||T^*| \right).
		   		   \end{equation} 
		   		   
		   Later Dragomir in \cite{dragomir} proved the following inequality for product of two operators, which assert that for $T, S \in \mathcal{B}(\mathcal{H})$ and $r \geq 1$ then, 
		   \begin{equation} \label{eq dragomir}
		   w^r(S^*T) \leq \frac{1}{2} \left\||T|^{2r} + |S|^{2r} \right\|.
		   \end{equation}
		   Recently, Al-Dolat et. al. in \cite{filomat} improved inequality (\ref{eq dragomir}), which say that for $T, S \in \mathcal{B}(\mathcal{H})$ and $\lambda \geq 0$ 
		   \begin{equation} \label{eq 4}
		   w^2(S^*T) \leq \frac{1}{2(1+\lambda)} \left\||T|^2 + |S|^2 \right\| w(S^*T) + \frac{\lambda}{2(1+\lambda)} \left\||T|^4 + |S|^4 \right\|.
		   \end{equation}
		   For further details of recent work on numerical radius inequalities, readers can see \cite{kstudiamath2, book, acta, palmo}.
		   The well-known Young inequality says that for any two positive real numbers $x, y$ and $t \in [0,1]$  then 
		   \begin{equation} \label{eq 7}
		   x^ty^{1-t} \leq tx + (1-t) y.
		   \end{equation} 
		   For $t = \frac{1}{2},$ we have the famous  arithmetic-geometric mean inequality, which says that for two positive real number $x, y$ 
		   \begin{equation}\label{eq am-gm}
		   \sqrt{xy} \leq \frac{x+y}{2}.
		   \end{equation}
		   The Kantorovich's ratio is defined by $$K(m, 2) = \frac{(m+1)^2}{4m}, $$ for a positive real number $m$. The Kantorovich ratio has the following properties: \\
		   
		   ~(1)~ $K(1,2) =1$ and $K(m, 2) = K(\frac{1}{m}, 2) >1.$\\
		   
		   ~(2)~$K(m, 2)$ is a monotone increasing function on the interval $[1, \infty)$ and monotone decreasing function on the interval $(0, 1].$ \\
		   
		   Zou et. all in \cite{jmi}  proved the following improvement of the famous Young inequality using Kantorovich's ratio, which says that for $x, y >0$ and $t \in [0,1]$,
		   \begin{equation}\label{eq 9}
		   K^r(m, 2) x^ty^{1-t} \leq tx +(1-t)y,
		   \end{equation}
		   where $r = \min \{t, 1-t \}.$ For further information readers are requested to go through \cite{mathann, laa}.  Nikzat et. al. in \cite{concr} obtain few improvements of numerical radius inequalities for invertible operator via Kantorovich ratio.\\
		   \noindent In this article, we obtain various numerical radius inequalities of operator and product of operators which generalize and improve on the bounds in (\ref{eq 1}), (\ref{eq 2}), (\ref{eq el kit}), (\ref{eq 3}), (\ref{eq dragomir}), (\ref{eq 4}) . In the last section we found a new improvement for numerical radius for invertible operator using Kantorovich ratio.
		\section{Main Results}
		In this section, we will prove various improvement of numerical radius upper bound. For this we need the following well-known lemmas. First result is a consequence of the spectral theorem along with Jensen's inequality.
		\begin{lemma}\label{positive op}\cite{mccarthy}
			Let $T \in \mathcal{B}(\mathcal{H})$ be a positive operator and $x$ be an unit vector in $\mathcal{H}.$ Then, for $r \geq 1,$ we have \[\langle Tx, x \rangle^r \leq \langle T^rx, x \rangle .\]
			
		\end{lemma}
	Our next lemma is the norm inequality of non-negative convex function.
	\begin{lemma}\label{convex op} \cite{aujla}
		Let $f$ be a non-negative convex function on $[0, \infty)$ and $A, B \in \mathcal{B}(\mathcal{H})$ be positive operators. Then we have \[\left\|f\left(\frac{A+B}{2}\right)\right\| \leq \left\|\frac{f(A) + f(B)}{2}\right\|. \]
		In particular if $r \geq 1$, then \[\left\|\left(\frac{A+B}{2} \right)^r \right\| \leq \left\|\frac{A^r+B^r}{2}\right\|. \]
	\end{lemma}
	The next result is a special case of mixed Schwarz inequality.
	\begin{lemma}\label{mixed schwarz} \cite{Res}
		Let $T \in \mathcal{B}(\mathcal{H})$. Then \[|\langle Tx, y \rangle|^2 \leq \langle |T|x, x \rangle \langle |T^*|y, y \rangle~~\forall ~x, y \in \mathcal{H} .\]
	\end{lemma}
Our following result is the famous Buzano generalization of Cauchy-Schwarz inequality.
\begin{lemma}\label{buzano}\cite{Buzano}
	Let $x, y, e \in \mathcal{H}$ with $\|e\|=1.$ Then we have \[|\langle x, e\rangle \langle e , y\rangle| \leq \frac{1}{2}\left(\|x\|\|y\| + |\langle x, y \rangle| \right) .\]
\end{lemma}
Recently Dolat et.al in \cite{filomat} proved a new improvement of Cauchy-Schwarz inequality. Our next lemma is the following:
\begin{lemma}\label{cauchyimp}
	Let $x, y \in \mathcal{H}.$ Then for any $\lambda \geq 0,$ we have \[|\langle x, y \rangle|^2 \leq \frac{1}{\lambda+1} \|x\|\|y\||\langle x, y\rangle| + \frac{\lambda}{\lambda + 1}\|x\|^2\|y\|^2 \leq \|x\|^2\|y\|^2. \]
\end{lemma}
Our next lemma is the operator version of classical Jensen;s inequality.
\begin{lemma}\label{jenson} \cite{Res}
	Let $T \in \mathcal{B}(\mathcal{H})$ be a self-adjoint operator whose spectrum contained in the interval $J$, and let $x \in \mathcal{H}$ be an unit vector. If $f$ be a convex function on $J$, then \[f(\langle Tx, x \rangle) \leq \langle f(T)x, x \rangle. \]
\end{lemma}
Our last lemma on this sequence is the following:
\begin{lemma} \label{convex}
	If $f: [0, a] \rightarrow [0, \infty)$, $a>0$ be an increasing convex function with $f(0) =0$ and $\alpha \in [0,1].$ Then \[f(\alpha x) \leq \alpha f(x) .\]
\end{lemma}
To prove our first theorem of the article we need the following consequence of the Buzano inequality.

\begin{lemma}\label{gen buzano}
	Let $x, y, e \in \mathcal{H}$ with $\|e\|=1$, then with $\lambda \geq 0$ we have, 
	\[|\langle x, e \rangle \langle e, y \rangle|^2 \leq \frac{1}{4} \left(\frac{2\lambda +3}{\lambda +1}\|x\|\|y\||\langle x, y \rangle| + \frac{2\lambda +1}{\lambda +1} \|x\|^2\|y\|^2 \right). \]
\end{lemma}
\begin{proof}
	Using Buzano generalization of Cauchy Schwarz inequality we get, 
	\begin{eqnarray*}
	|\langle x, e \rangle \langle e, y \rangle|^2 &\leq& \frac{1}{4} \left( \|x\|^2\|y\|^2 + 2 |\langle x, y \rangle| \|x\|\|y\| + |\langle x, y \rangle|^2\right)\\
	&\leq& \frac{1}{4}\left(\|x\|^2\|y\|^2 + 2 |\langle x, y \rangle| \|x\|\|y\| + \frac{1}{\lambda +1} \|x\|\|y\||\langle x, y \rangle|\right)\\&& + \frac{\lambda}{4(\lambda +1)} \|x\|^2\|y\|^2~~~(\mbox{using Lemma \ref{cauchyimp}}) \\
	&=&  \frac{1}{4} \left(\frac{2\lambda +3}{\lambda +1}\|x\|\|y\||\langle x, y \rangle| + \frac{2\lambda +1}{\lambda +1} \|x\|^2\|y\|^2 \right).
	\end{eqnarray*}
\end{proof}
Now we are ready to prove our first theorem of the article. 
\begin{theorem}\label{th2}
	Let $T \in \mathcal{B}(\mathcal{H})$ and $\lambda \geq0.$ Then \[w^4(T) \leq \frac{2\lambda +3}{8(\lambda + 1)} \left\||T|^2 + |T^*|^2\right\|w(T^2) + \frac{2\lambda +1}{8(\lambda +1)}\left\||T|^4 + |T^*|^4\right\|. \] 
\end{theorem}
\begin{proof}
	Let $x \in \mathcal{H}$ with $\|x\| =1$. Replacing $x$ by $Tx$, $y$ by $T^*x$ and $e$ by $x$ in\\ Lemma \ref{gen buzano} we get,
	\begin{eqnarray*}
	|\langle Tx, x \rangle|^4 &\leq& \frac{1}{4}\left( \frac{2\lambda +3}{\lambda +1} \|Tx\|\|T^*x\|\left|\langle T^2x, x \rangle\right| +\frac{2\lambda +1}{\lambda +1} \|Tx\|^2\|T^*x\|^2   \right)\\
	&\leq&  \frac{2\lambda +3}{8(\lambda +1)}\left(\|Tx\|^2 + \|T^*x\|^2\right) |\langle T^2x, x \rangle| +\frac{2\lambda +1}{8(\lambda +1)}\left(\|Tx\|^4 + \|T^*x\|^4 \right)  \\
	&\leq& \frac{2\lambda +3}{8(\lambda +1)} \left\langle (|T|^2 + |T^*|^2)x, x \right\rangle |\langle T^2x, x \rangle| + \frac{2\lambda +1}{8(\lambda +1)} \left\langle(|T|^4 + |T^*|^4)x, x  \right\rangle \\
	&\leq&  \frac{2\lambda +3}{8(\lambda + 1)} \left\||T|^2 + |T^*|^2\right\|w(T^2) + \frac{2\lambda +1}{8(\lambda +1)}\left\||T|^4 + |T^*|^4\right\|.
	\end{eqnarray*}
By taking supremum over all unit vectors $x \in \mathcal{H}$, we get our required inequality.
\end{proof}
\begin{cor}\label{imp dragomir}
	Let $T \in \mathcal{B}(\mathcal{H})$ and $\lambda \geq 0.$ Then 
	\begin{eqnarray*}
	w^4(T) &\leq& \frac{2\lambda +3}{8(\lambda + 1)} \left\||T|^2 + |T^*|^2\right\|w(T^2) + \frac{2\lambda +1}{8(\lambda +1)}\left\||T|^4 + |T^*|^4\right\|\\
	&\leq& \frac{1}{2} \left\||T|^4 + |T^*|^4 \right\|.
	\end{eqnarray*}
\end{cor}
\begin{proof}
	By Theorem \ref{th2} we have 
	\begin{eqnarray*}
	w^4(T) &\leq& \frac{2\lambda +3}{8(\lambda + 1)} \left\||T|^2 + |T^*|^2\right\|w(T^2) + \frac{2\lambda +1}{8(\lambda +1)}\left\||T|^4 + |T^*|^4\right\|\\
	&\leq&  \frac{2\lambda +3}{16(\lambda + 1)} \left\||T|^2 + |T^*|^2\right\| ^2 + \frac{2\lambda +1}{8(\lambda +1)}\left\||T|^4 + |T^*|^4\right\|, (\mbox{using inequality (\ref{eq dragomir})})\\
	&=& \frac{2\lambda +3}{16(\lambda + 1)} \left\|\left(|T|^2 + |T^*|^2\right)^2\right\|  + \frac{2\lambda +1}{8(\lambda +1)}\left\||T|^4 + |T^*|^4\right\|\\
	&\leq& \frac{2\lambda +3}{8(\lambda + 1)} \left\||T|^4 + |T^*|^4\right\|  + \frac{2\lambda +1}{8(\lambda +1)}\left\||T|^4 + |T^*|^4\right\|~(\mbox{using Lemma \ref{convex op}})\\
	&=& \frac{1}{2} \left\||T|^4 + |T^*|^4 \right\|.
	\end{eqnarray*}
So our obtained inequality in Theorem \ref{th2} gives an improvement of the inequality (\ref{eq dragomir}) for $r=2.$
\end{proof}
If we take $\lambda =1$ in Theorem \ref{th2} we get the following corollary which is a refinement of \cite[Th. 2.1]{omidvar}
\begin{cor}  \cite[Remark 3.4]{klaa}
	Let $T \in \mathcal{B}(\mathcal{H})$, then 	
	\begin{eqnarray*}
	w^4(T) &\leq& \frac{5}{16} \left\||T|^2 + |T^*|^2 \right\|w(T^2) + \frac{3}{16} \left\||T|^4 + |T^*|^4 \right\| \\
	&\leq&   \frac{1}{8} \left\||T|^2 + |T^*|^2 \right\|w(T^2) + \frac{3}{8} \left\||T|^4 + |T^*|^4 \right\|\\
	&\leq& \frac{1}{2} \left\||T|^4 + |T^*|^4 \right\|.
	\end{eqnarray*}
\end{cor}

Next we prove the following lemma the proof of which follows from Lemma \ref{positive op} and Lemma \ref{buzano} .
\begin{lemma} \label{dolat 23}
	Let $x, y, e \in \mathcal{H}$ with $\|e\|=1.$ Then for $\alpha \in [0,1]$ and $r \geq 1, $ we have 
	\[|\langle x, e \rangle \langle e, y \rangle|^{2r} \leq \frac{\alpha}{2} \left[\|x\|^{2r}\|y\|^{2r} + |\langle x, y \rangle|^{2r} \right] + \frac{1-\alpha}{2} \left[\|x\|^r\|y\|^r + |\langle x, y \rangle|^r \right]|\langle x, e \rangle \langle e, y \rangle|^r. \]
\end{lemma}
Using the above lemma, we prove our next theorem.
\begin{theorem}\label{th4}
	Let $T \in \mathcal{B}(\mathcal{H}).$ Then for $\alpha \in [0,1]$ and $r \geq 1$ we get
	\begin{eqnarray*}
	w^{4r}(T) &\leq& \frac{\alpha}{8} \left\||T|^{4r} + |T^*|^{4r} \right\| + \frac{\alpha}{4}  w(|T^*|^{2r} |T|^{2r})  + \frac{\alpha}{2} w^{2r}(T^2)\\ && +\frac{1-\alpha}{4} \left\| |T|^{2r} + |T^*|^{2r}\right\| w^{2r}(T) + \frac{1-\alpha}{2}w^r(T^2) w^{2r}(T)\\
&\leq&	\frac{1}{2}\left\||T|^{4r} + |T^*|^{4r}\right\|.
	\end{eqnarray*}
\end{theorem}

\begin{proof}
	Let $x \in \mathcal{H}$ with $\|x\|=1.$
	Replacing $x$ by $Tx$, $y$ by $T^*x$ and $e$ by $x$ in Lemma \ref{dolat 23} we get, 
	 \begin{eqnarray*}
	|\langle Tx, x \rangle|^{4r} &\leq& \frac{\alpha}{2} \left[\|Tx\|^{2r} \|T^*x\|^{2r} + |\langle T^2x, x \rangle|^{2r} \right]\\ && + \frac{1-\alpha}{2}\left[\|Tx\|^r \|T^*x\|^{r} +|\langle T^2x, x \rangle|^r \right]|\langle Tx, x \rangle|^{2r}\\
	&\leq& \frac{\alpha}{2} \left[\langle |T|^{2r} x, x \rangle \langle x, |T^*|^{2r}x \rangle + |\langle T^2x, x \rangle|^{2r} \right]\\ && + \frac{1-\alpha}{2} \left[\frac{\|Tx\|^{2r} + \|T^*x\|^{2r}}{2} + |\langle T^2x, x \rangle|^r\right] |\langle Tx, x \rangle|^{2r}\\
	&&\,\,\,\,\,\,\,\,\,\,\,\,\,\,\,\,\,\,\,\,\,\,\,\,\,\,\,\,\,\,\,\,\,\,\,\,\,\,\,\,\,\,\,\,\,\,\,\,\,\,\,\,\,\,\,\,\,\,\,(\mbox{using Lemma \ref{positive op} and inequality (\ref{eq am-gm})} )\\
	&\leq& \frac{\alpha}{2} \left[\frac{\||T|^{2r}x\|\||T^*|^{2r}x\| + |\langle |T|^{2r}x, |T^*|^{2r}x \rangle|}{2} + |\langle T^2x, x \rangle|^{2r} \right] \\ && + \frac{1-\alpha}{2} \left[ \frac{1}{2} \left\langle \left(|T|^{2r} + |T^*|^{2r} \right)x, x  \right\rangle + |\langle T^2x, x \rangle|^r\right] |\langle Tx, x \rangle|^{2r}\\
	&&\,\,\,\,\,\,\,\,\,\,\,\,\,\,\,\,\,\,\,\,\,\,\,\,\,\,\,\,\,\,\,\,\,\,\,\,\,\,\,\,\,\,\,\,\,\,\,\,\,\,\,\,\,\,\,\,\,\,\,\,\,\,\,\,\,\,\,\,\,\,\,\,\,(\mbox{using Lemma \ref{buzano} and Lemma \ref{positive op}})\\
	&\leq& \frac{\alpha}{8} \left\langle \left(|T|^{4r} + |T^*|^{4r} \right) x,x \right \rangle + \frac{\alpha}{4} \left|\left\langle\left(|T^*|^{2r} |T|^{2r} \right)x, x \right\rangle \right|\\ && + \frac{\alpha}{2} |\langle T^2x, x \rangle|^{2r} + \frac{1-\alpha}{4} \left\langle \left( |T|^{2r} + |T^*|^{2r}\right)x, x \right\rangle |\langle Tx, x \rangle|^{2r} \\ &&+ \frac{1-\alpha}{2} |\langle T^2x, x \rangle|^r |\langle Tx, x \rangle|^{2r}~(\mbox{using inequality (\ref{eq am-gm})})\\
	&\leq& \frac{\alpha}{8} \left\||T|^{4r} + |T^*|^{4r} \right\| + \frac{\alpha}{4} w\left(|T^*|^{2r} |T|^{2r} \right) + \frac{\alpha}{2} w^{2r}(T^2) \\ && + \frac{1-\alpha}{4} \left\||T|^{2r} + |T^*|^{2r} \right\| w^{2r}(T) + \frac{1-\alpha}{2} w^r(T^2) w^{2r}(T).
	\end{eqnarray*}

Now taking supremum over $x$ with $\|x\|=1$ we obtain our first inequality.
For the second inequality we see, 
\begin{eqnarray*}
w^{4r}(T) &\leq& \frac{\alpha}{8} \left\||T|^{4r} + |T^*|^{4r} \right\| + \frac{\alpha}{8} \left\||T|^{4r} + |T^*|^{4r} \right\| \\ && + \frac{\alpha}{4} \left\||T|^{4r} + |T^*|^{4r} \right\| + \frac{1-\alpha}{8} \left\||T|^{2r} + |T^*|^{2r} \right\|^2 \\ && + \frac{1-\alpha}{8}  \left\||T|^{2r} + |T^*|^{2r} \right\|^2~(\mbox{using inequality (\ref{eq dragomir}}))\\
&=& \frac{\alpha}{2} \left\||T|^{4r} + |T^*|^{4r} \right\| + \frac{1-\alpha}{4} \left\|\left(|T|^{2r} + |T^*|^{2r}\right)^2 \right\|\\
&\leq& \frac{\alpha}{2} \left\||T|^{4r} + |T^*|^{4r} \right\| + \frac{1-\alpha}{2} \left\||T|^{4r} + |T^*|^{4r} \right\|~~(\mbox{using Lemma \ref{convex op}})\\
&=& \frac{1}{2} \left\||T|^{4r} + |T^*|^{4r} \right\|.
\end{eqnarray*}
\end{proof}
Next, we prove our first inequality on product of two operators.
\begin{theorem}\label{th 6}
	Let $T, S \in \mathcal{B}(\mathcal{H}).$ For $r \geq 1$ and $\lambda \geq 0$ then we have
	\begin{eqnarray*}
		w^{2r}(T^*S) &\leq& \frac{1}{2(\lambda+1)}\left\||T|^{2r} + |S|^{2r}\right\|w^r(T^*S) + \frac{\lambda}{4(\lambda+1)} \left\||T|^{4r} + |S|^{4r}\right\|\\&&+ \frac{\lambda}{2(\lambda +1)}w\left(|S|^{2r}|T|^{2r} \right). 
	\end{eqnarray*}
\end{theorem}
\begin{proof}
	Let $x \in \mathcal{H}$ with $\|x\|=1.$ Then 
	\begin{eqnarray*}
		\left|\langle T^*Sx, x \rangle \right|^{2r} &=& \left|\langle Tx, Sx \rangle \right|^{2r}\\
		&\leq& \left(\frac{1}{\lambda+1} \|Tx\|\|Sx\| |\langle Tx, Sx \rangle| + \frac{\lambda}{\lambda+1} \|Tx\|^2\|Sx\|^2 \right)^r (\mbox{using Lemma \ref{cauchyimp}})\\
		&\leq& \frac{1}{\lambda+1} \|Tx\|^{r} \|Sx\|^r |\langle T^*Sx, x \rangle|^r + \frac{\lambda}{\lambda+1} \|Tx\|^{2r} \|Sx\|^{2r}\\
		&&\,\,\,\,\,\,\,\,\,\,\,\,\,\,\,\,\,\,\,\,\,\,\,\,\,\,\,\,\,\,\,\,\,\,\,\,\,\,\,\,\,\,\,\,\,\,\,\,~(\mbox{using the convexity of the function $f(t) = t^r$ })\\
		&\leq& \frac{1}{2(\lambda+1)} \left\langle \left(|T|^{2r}+ |S|^{2r} \right)x, x \right\rangle \left|\langle T^*Sx, x \rangle \right|^r + \frac{\lambda}{\lambda+1} \langle |T|^{2r}x, x \rangle \langle x, |S|^{2r}x \rangle \\ 
		&\leq& \frac{1}{2(\lambda+1)} \left\langle \left(|T|^{2r}+ |S|^{2r} \right)x, x \right\rangle \left|\langle T^*Sx, x \rangle \right|^r \\&&+ \frac{\lambda}{2(\lambda+1)} \left(\left\||T|^{2r}x\right\|\left\||S|^{2r}x \right\| + \left\langle |T|^{2r}x, |S|^{2r}x \right\rangle  \right)~(\mbox{using Lemma \ref{buzano}})\\
		&\leq& \frac{1}{2(\lambda+1)} \left\langle \left(|T|^{2r}+ |S|^{2r} \right)x, x \right\rangle \left|\langle T^*Sx, x \rangle \right|^r \\&&+ \frac{\lambda}{4(\lambda+1)} \left\langle \left(|T|^{4r} + |S|^{4r}  \right)x, x \right\rangle + \frac{\lambda}{2(\lambda+1)} \left\langle \left(|S|^{2r} |T|^{2r}\right)x, x \right\rangle\\
		&\leq& \frac{1}{2(\lambda+1)} \left\| |T|^{2r}+ |S|^{2r} \right \| w^r(T^*S) +\frac{\lambda}{4(\lambda+1)} \left\| |T|^{4r} + |S|^{4r} \right\| \\&&+ \frac{\lambda}{2(\lambda+1)} w\left(|S|^{2r}|T|^{2r} \right).
	\end{eqnarray*}
	Now taking supremum over $x$ with $\|x\|=1$ we get our required inequality.

\end{proof}
The following corollaries are obtained from Theorem \ref{th 6}.
\begin{cor}\label{aldolat filomat}
	Let $T, S \in \mathcal{B}(\mathcal{H})$, then 
	\begin{eqnarray*}
		w^2(T^*S) &\leq& \frac{1}{2(\lambda+1)}\left\||T|^2 + |S|^2 \right\| w(T^*S) + \frac{\lambda}{4(\lambda+1)} \left\| |T|^{4} + |S|^4 \right\|\\ && + \frac{\lambda}{2(\lambda+1)} w\left( |S|^2|T|^2\right)\\
		&\leq& \frac{1}{2(\lambda+1)} \left\||T|^2 + |S|^2 \right\| w(T^*S) + \frac{\lambda}{2(\lambda+1)} \left\||T|^4 + |S|^4 \right\|.
	\end{eqnarray*}
\end{cor}
\begin{proof}
	Taking $r=1$ in Theorem \ref{th 6} we get the first inequality and second follows from inequality (\ref{eq dragomir}). So our inequality in Theorem \ref{th 6} generalizes and improves the inequality obtain by Al-Dolat et al in \cite[Th. 2.6]{filomat}.
\end{proof}
\begin{cor}
Let $T, S \in \mathcal{B}(\mathcal{H})$. Then for any $\lambda \geq 0$ and $r \geq 1,$
\begin{eqnarray*}
	w^{2r}(T^*S) &\leq& \frac{1}{2(\lambda+1)}\left\||T|^{2r} + |S|^{2r}\right\|w^r(T^*S) + \frac{\lambda}{4(\lambda+1)} \left\||T|^{4r} + |S|^{4r}\right\|\\&&+ \frac{\lambda}{2(\lambda+1)}w\left(|S|^{2r}|T|^{2r} \right) \\
	&\leq& \frac{1}{2} \left\||T|^{4r} + |S|^{4r} \right\|.
\end{eqnarray*}
\end{cor}
\begin{proof}
	Using inequality (\ref{eq dragomir}) we get
	\begin{eqnarray*}
		w^{2r}(T^*S) &\leq& \frac{1}{2(\lambda+1)}\left\||T|^{2r} + |S|^{2r}\right\|w^r(T^*S) + \frac{\lambda}{4(\lambda+1)} \left\||T|^{4r} + |S|^{4r}\right\|\\&&+ \frac{\lambda}{2(\lambda+1)}w\left(|S|^{2r}|T|^{2r} \right) \\
		&\leq& 	 \frac{1}{4(\lambda+1)}\left\||T|^{2r} + |S|^{2r}\right\|^2 + \frac{\lambda}{4(\lambda+1)} \left\||T|^{4r} + |S|^{4r}\right\|\\&&+ \frac{\lambda}{4(\lambda+1)} \left\||T|^{4r} + |S|^{4r} \right\|~(\mbox{using inequality (\ref{eq dragomir})})\\
		&\leq&  \frac{1}{4(\lambda+1)}\left\||T|^{4r} + |S|^{4r}\right\| + \frac{\lambda}{4(\lambda+1)} \left\||T|^{4r} + |S|^{4r}\right\|\\&&+ \frac{\lambda}{4(\lambda+1)} \left\||T|^{4r} + |S|^{4r} \right\|~(\mbox{using Lemma \ref{convex op}})\\
		&=& \frac{1}{2} \left\| |T|^{4r} + |S|^{4r}\right\|.
	\end{eqnarray*}
\end{proof}

Next, we prove an improvement of triangle inequality of norm. 
\begin{lemma}\label{imp triangle}
	Let $x, y \in  \mathcal{H}$ and $\lambda \geq 0.$ Then 
	\begin{eqnarray*}
		\|x+y\|^2 &\leq& \frac{1}{\lambda+1} \left(\|x\|+\|y\| \right)\|x+y\| + \frac{\lambda}{\lambda+1} \left(\|x\|+\|y\|\right)^2 \\
		&\leq& \left(\|x\|+\|y\| \right)^2.
	\end{eqnarray*}
\end{lemma}
\begin{proof}
	Using triangle inequality for norm we get,
	\begin{eqnarray*}
		\|x+y\|^2 &\leq& \|x+y\|\left(\|x\|+\|y\|\right)\\
		&\leq&  \|x+y\|\left(\|x\|+\|y\|\right) + \lambda\left(\left(\|x\|+\|y\| \right)^2-\|x+y\|^2 \right)\\
		\Rightarrow (\lambda+1)\|x+y\|^2  &\leq& \left(\|x\|+\|y\|\right)\|x+y\| +  \lambda \left(\|x\|+\|y\|\right)^2 \\
		\|x+y\|^2 &\leq& \frac{1}{\lambda+1} \left(\|x\|+\|y\| \right)\|x+y\| +\frac{\lambda}{\lambda+1} \left(\|x\|+\|y\| \right)^2 . 
	\end{eqnarray*}
	Using the fact that $\|x+y\| \leq \|x\|+ \|y\|$ we can prove the second inequality of the lemma.
\end{proof}
Next lemma is the Polarization identity for real Hilbert space.
\begin{lemma}\label{polarization}
	Let $x, y \in \mathbb{H},$ where $\mathbb{H}$ is a real Hilbert space. Then \[\langle x, y \rangle = \frac{1}{4}\left(\|x+y\|^2 - \|x-y\|^2 \right). \]
\end{lemma}
Using Lemma \ref{imp triangle} and Lemma \ref{polarization} we get the following corollary.
\begin{cor}\label{app imp triangle}
	Let $x, y \in \mathbb{H}$ and $\lambda \geq 0$. Then \[|\langle x, y \rangle| \leq \frac{\lambda}{4(\lambda+1)} \left(\|x\| + \|y\| \right)^2 + \frac{1}{4(\lambda +1)} \|x+y\| \left(\|x\| +\|y\| \right) .\]
\end{cor}
Next proposition is a simple consequence of the above corollary. 
\begin{prop}\label{polrization app}
	Let $	T, S \in \mathcal{B}(\mathbb{H})$, and $\lambda \geq 0.$ Then 
	\begin{eqnarray*}
		w(T^*S) &\leq& \frac{\lambda}{2(\lambda+1)} \left\||T|^2 + |S|^2 \right\| + \frac{1}{4(\lambda+1)} \|T+S\| \left(\|T\|+\|S\| \right).
	\end{eqnarray*}
\end{prop}
\begin{proof}
	Let $x \in \mathbb{H}$ with $\|x\|=1$. Replacing  $x$ by $Tx$ and $y$ by $Sx$ in corollary \ref{app imp triangle} we get
	\begin{eqnarray*}
		|\langle T^*Sx, x \rangle| &\leq& \frac{\lambda}{4(\lambda+1)} \left( \|Tx\| + \|Sx\|\right)^2 + \frac{1}{4(\lambda+1)} \left\|(T+S)x\right\| \left(\|Tx\| + \|Sx\| \right)\\
		&\leq& \frac{\lambda}{2(\lambda+1)} \left(\|Tx\|^2 + \|Sx\|^2 \right) +\frac{1}{4(\lambda+1)} \left\|(T+S)x\right\| \left(\|Tx\| + \|Sx\| \right)\\
		&&\,\,\,\,\,\,\,\,\,\,\,\,\,\,\,\,\,\,\,\,\,\,\,\,\,\,\,\,\,\,\,\,\,\,\,\,\,\,\,\,\,\,\,\,\,\,\,\,\,\,\,\,\,\,\,\,\,\,\,\,\,\,\,\,\,\,\,\,\,\,\,~(\mbox{using the convexity of $f(t) =t^r$})\\
		&\leq& \frac{\lambda}{2(\lambda+1)} \left\langle \left(|T|^2 + |S|^2\right)x, x \right\rangle + \frac{1}{4(\lambda+1)} \left\|T+S\right\| \left(\|T\| + \|S\| \right)\\
		&\leq& \frac{\lambda}{2(\lambda+1)} \left\||T|^2 + |S|^2 \right\| + \frac{1}{4(\lambda+1)} \|T+S\| \left(\|T\|+\|S\| \right).
	\end{eqnarray*}
	Now taking supremum over all unit vectors $x$ we obtain our required result. 
\end{proof}

Next, we prove an upper bound involving sum and product of operators.

\begin{theorem}\label{th5}
	Let $A, B, C, D \in \mathcal{B}(\mathcal{H}).$ Then for $r \geq 1$ we have 
	\begin{eqnarray*}
	w^{2r}(A^*B+ C^*D) &\leq& 2^{2r-3} \left\||A|^{4r} + |B|^{4r} + |C|^{4r} + |D|^{4r} \right\|\\&& + 2^{2r-2} \left(w^r\left(|B|^2|A|^2 \right) +w^r\left(|D|^2|C|^2 \right) \right). 
	\end{eqnarray*}
\end{theorem}

\begin{proof}
	Let $x \in \mathcal{H}$ with $\|x\|=1.$ Then we have
	\begin{eqnarray*}
	\left|\langle (A^*B+ C^*D)x, x \rangle \right|^{2r} &\leq& \left(|\langle A^*Bx, x \rangle| + |\langle C^*Dx , x \rangle| \right)^{2r}\\
	&\leq& 2^{2r-1} \left(|\langle A^*Bx, x \rangle|^{2r}+ |\langle C^*Dx , x \rangle|^{2r} \right)\\
	&\leq& 2^{2r-1} \left(\langle |A|^2x, x \rangle^r \langle |B|^2x, x \rangle^r  + \langle |C|^2x, x \rangle^r \langle |D|^2x, x \rangle^r  \right)\\
	&=& 2^{2r-1} \left(\left(\langle |A|^2x, x \rangle \langle x, |B|^2x \rangle \right)^r + \left(\langle |C|^2x, x \rangle \langle x, |D|^2x \rangle \right)^r  \right)\\
	&\leq& 2^{2r-1} \left(\frac{\||A|^2x\|\cdot \||B|^2x\| + \left|\langle |A|^2x, |B|^2x \rangle\right|}{2} \right)^r\\&& + 2^{2r-1} \left(\frac{\||C|^2x\|\cdot \||D|^2x\| +\left| \langle |C|^2x, |D|^2x \rangle\right|}{2} \right)^r 
		\end{eqnarray*}
\begin{eqnarray*}
	&\leq& 2^{2r-2} \left(\||A|^2x\|^r\cdot \||B|^2x\|^r + \left|\langle |A|^2x, |B|^2x \rangle\right|^r \right)\\ && + 2^{2r-2} \left(\||C|^2x\|^r\cdot \||D|^2x\|^r + \left|\langle |C|^2x, |D|^2x \rangle\right|^r \right)\\
	&\leq& 2^{2r-3} \left\langle \left(|A|^{4r} + |B|^{4r} + |C|^{4r} + |D|^{4r} \right)x, x \right\rangle \\ &&+ 2^{2r-2} \left(\left|\left\langle \left(|B|^2 |A|^2\right)x, x \right\rangle \right|^r + \left|\left\langle \left(|D|^2 |C|^2\right)x, x \right\rangle \right|^r \right)\\
	&\leq&  2^{2r-3} \left\||A|^{4r} + |B|^{4r} + |C|^{4r} + |D|^{4r} \right\|\\&&+ 2^{2r-2} \left(w^r \left(|B|^2|A|^2 \right) + w^r \left(|D|^2|C|^2 \right) \right).
		\end{eqnarray*}
	Taking supremum over all unit vectors we get our desired inequality.
\end{proof}
The following corollary can be obtained from Theorem \ref{th5}. 
\begin{cor}
	If $A, B \in \mathcal{B}(\mathcal{H})$ and $r \geq 2$ then 
	\begin{eqnarray*}
	w^{2r}(AB+CD) & \leq &2^{2r-3} \left\||A^*|^{4r} + |B|^{4r} + |C^*|^{4r} + |D|^{4r}\right\| \\&&+ 2^{2r-2}\left[ w^r\left(|B|^2|A^*|^2 \right) + w^r\left(|D|^2|C^*|\right)\right] \\
	&\leq& 2^{2r-2}\left[ \left\||A^*|^{4r} + |B|^{4r} \right\| + \left\||C^*|^{4r} + |D|^{4r} \right\|\right.]
	\end{eqnarray*}
\end{cor}

\section{Inequalities via Kantorovich ratio}
In this section we provide an upper bound for numerical radius of invertible operator via Kantorovich ratio. Few preliminary things we have already covered on Kantorovich ratio in the introductory part of the article. So, we now move on to prove our first theorem of this section.
\begin{theorem}\label{th1}
	Let $T\in \mathcal{B}(\mathcal{H})^{-1}$ and let $f$ be an increasing convex function on $[0, \infty)$ with $f(0) =0.$ If $m|T| \leq |T^*|$ or $m|T^*| \leq |T|$ Then for $m>1$, \[f\left(w^2(T) \right) \leq \frac{1}{4\sqrt{K(m, 2)}} \left\|f(|T|^2) + f(|T^*|^2)\right\|+ \frac{1}{2} f(w(T^2)) ,\] where $K(m, 2) = \frac{(m+1)^2}{4m}.$	
\end{theorem}
\begin{proof}
	Taking $x=Tu, y=T^*u, e=u$ with $\|u\|=1$ in Lemma \ref{buzano} and applying the monotonocity of the function $f$ we get,
	\begin{eqnarray*}
		f\left( |\langle Tu, u \rangle|^2\right) &\leq& \frac{1}{2}f\left(\|Tu\|\|T^*u\| + |\langle T^2u, u \rangle| \right)~(\mbox{using Lemma \ref{convex}}) \\
		&\leq& \frac{1}{2}\left[f(\|Tu\|\|T^*u\|) + f(|\langle T^2u, u \rangle|)\right]\\
		&=& \frac{1}{2}f\left(\langle |T|^2u, x \rangle^{1/2}\langle |T^*|^2u, x \rangle^{1/2} \right)+ \frac{1}{2}f(|\langle T^2u, u \rangle|)\\
		&\leq& f\left(\frac{1}{4\sqrt{K(m,2)}}\left\langle(|T|^2+ |T^*|^2)u, u\right\rangle \right) + \frac{1}{2} f\left(|\langle T^2u, u \rangle| \right)\\
			&&\,\,\,\,\,\,\,\,\,\,\,\,\,\,\,\,\,\,\,\,\,\,\,\,\,\,\,\,\,\,\,\,\,\,\,\,\,\,\,\,\,\,\,\,\,\,\,\,\,\,\,\,\,\,\,\,\,\,\,\,\,\,\,(\mbox{using inequality (\ref{eq 9})})
			\end{eqnarray*}
	\begin{eqnarray*}
		&\leq& \frac{1}{4\sqrt{K(m,2)}} \left(f(\langle|T|^2x, x \rangle\right) + f\left(\langle |T^*|^2u,u \rangle \right)+ \frac{1}{2} f\left( |\langle T^2u, u \rangle|\right)	 \\
		&=& \frac{1}{4\sqrt{K(m,2)}} \left\langle(f(|T|^2) + f(|T^*|^2)u, u  \right\rangle+ \frac{1}{2} f\left( |\langle T^2u, u \rangle|\right)\\
		&\leq& \frac{1}{4\sqrt{K(m,2)}} \left\|f(|T^2|) + f(|T^*|^2) \right\| + \frac{1}{2} f(w(T^2)).
	\end{eqnarray*}
	where $m= \frac{\langle |T|^2u, u \rangle}{\langle |T^*|^2u, u \rangle}$ and $K(m, 2) = \frac{(m+1)^2}{4m}.$
\end{proof}
If we take $f(t) = t^r$ with $r \geq 1$ in Theorem \ref{th1}, we obtain the following  inequality which is clearly an improvement of the inequality (\ref{eq 3}) for invertible operator.
\begin{cor}
	Let $T \in \mathcal{B}(\mathcal{H})^{-1}.$ If $m|T| \leq |T^*|$ or $m|T^*| \leq |T|$ with $m>1,$ we get \[ w^{2r}(T) \leq \frac{1}{4\sqrt{K(m,2)}} \left\||T^{2r} + |T^*|^{2r}| \right\| + \frac{1}{2} w^r(T^2).\]
\end{cor}

\begin{lemma} \label{ext buzano}
	Let $x, y, e \in \mathcal{H}$ with $\|e\|=1.$ Then for $r\geq 1$ and $\lambda \in [0,1]$, we have 
	\[|\langle x, e \rangle \langle e, y \rangle|^r \leq \beta \|x\|^r \|y\|^r + \lambda |\langle x, y \rangle|^r \] where $\beta = \min \left\{ \frac{1+ \lambda}{2}, 1- \frac{\lambda}{2}\right\}$ and $\gamma = \min \left\{\frac{1-\lambda}{2}, \frac{\lambda}{2} \right\}.$
\end{lemma}
\begin{proof}
	Using the convexity of the function $f(t) =t^r,$ where $r \geq 1$ we get,
	\begin{eqnarray*}
		|\langle x, e \rangle \langle e, y \rangle| &=& \lambda |\langle x, e \rangle \langle e, y \rangle| + (1-\lambda) |\langle x, e \rangle \langle e, y \rangle|\\
		&\leq& \lambda \|x\|\|y\| + \frac{1-\lambda}{2} \left[\|x\|\|y\| + |\langle x, y \rangle| \right]~~(\mbox{using Lemma \ref{buzano}})\\
		\Rightarrow |\langle x, e \rangle \langle e, y \rangle|^r &\leq& \lambda \|x\|^r\|y\|^r + (1-\lambda) \left[\frac{\|x\|\|y\| + |\langle x, y \rangle|}{2} \right]^r\\
		&\leq& \lambda \|x\|^r\|y\|^r + \frac{1-\lambda}{2} \left[\|x\|^r\|y\|^r + ||\langle x, y \rangle |^r \right]\\
		&=& \frac{1+\lambda}{2} \|x\|^r\|y\|^r + \frac{1-\lambda}{2}|\langle x, y \rangle |^r.
	\end{eqnarray*}
	So, we have,
	\begin{equation}\label{buzeq1}
	|\langle x, e \rangle \langle e, y \rangle|^r \leq  \frac{1+\lambda}{2} \|x\|^r\|y\|^r + \frac{1-\lambda}{2}|\langle x, y \rangle |^r.
	\end{equation}
	Similarly, we get,
	\begin{equation}\label{buzeq2}
	|\langle x, e \rangle \langle e, y \rangle|^r \leq \left(1-\frac{\lambda}{2}\right)\|x\|^r\|y\|^r + \frac{\lambda}{2} |\langle x, y \rangle|^r.
	\end{equation}
	From inequalities (\ref{buzeq1}) and (\ref{buzeq2}) result follows.
\end{proof}

\begin{prop} \label{kantorovich 1}
	Let $T \in \mathcal{B}(\mathcal{H})^{-1}$. Then for $m, r \geq 1,$ and If $m|T| \leq |T^*|$ or $m|T^*| \leq |T|$
	\[w^r(T) \leq \frac{\beta}{2 \sqrt{K(m,2)}} \left\||T|^2 + |T^*|^2 \right\|  + \gamma w(T^2)\]
	where $\beta$ and $\gamma$ are as in Lemma \ref{ext buzano}, $m= \frac{\langle |T|^2u, u \rangle}{\langle |T^*|^2u, u \rangle}$ and $K(m, 2) = \frac{(m+1)^2}{4m}.$
\end{prop}
\begin{proof}
	The proof follows from Lemma \ref{ext buzano} replacing $x$ by $Tx$, $y$ by $T^*x$,~ $e$ by $x$ and proceeding similarly like in Theorem \ref{th1}.
\end{proof}

			\bibliographystyle{amsplain}
			
\end{document}